\documentclass[12pt]{amsart}

\usepackage{graphicx,amssymb,latexsym,amsfonts,txfonts,amsmath,amsthm}
\usepackage{pdfsync,color,tabularx,rotating}
\usepackage[all,cmtip]{xy}

\usepackage{tikz}

\advance\hoffset by -0.9 truecm

\newtheorem{theo}{Theorem}

\newtheorem{cor}[theo]{Corollary}

\theoremstyle{remark}
\newtheorem{rema}[theo]{\bf Remark}

\theoremstyle{remark}
\newtheorem{exam}[theo]{\bf Example}

\theoremstyle{remark}

\begin{document}

\title{Automorphism groups of maps, hypermaps and dessins}

\author{Gareth A. Jones}

\address{School of Mathematics, University of Southampton, Southampton SO17 1BJ, UK}
\email{G.A.Jones@maths.soton.ac.uk}

\maketitle


\begin{abstract}
A detailed proof is given of a theorem describing the centraliser of a transitive permutation group, with applications to automorphism groups of objects in various categories of maps, hypermaps, dessins, polytopes and covering spaces, where the automorphism group of an object is the centraliser of its monodromy group.  An alternative form of the theorem, valid for finite objects, is discussed, with counterexamples based on Baumslag--Solitar groups to show how it fails more generally. The automorphism groups of objects with primitive monodromy groups are described, as are those of non-connected objects.
\end{abstract}


\section{Introduction}\label{intro}

In certain categories $\mathfrak C$, such as those consisting of maps or hypermaps, oriented or unoriented, or of dessins d'enfants (regarded as finite oriented hypermaps), each object $\mathcal O$ can be identified with a permutation representation $\theta:\Gamma\to S:={\rm Sym}(\Omega)$ of a `parent group' $\Gamma=\Gamma_{\mathfrak C}$ on some set $\Omega$, and the morphisms ${\mathcal O}_1\to{\mathcal O_2}$ can be identified with the functions $\Omega_1\to\Omega_2$ which commute with the actions of $\Gamma$ on the corresponding sets $\Omega_1$ and $\Omega_2$. These are the `permutational categories' defined and discussed in~\cite{Jon16}. The automorphism group ${\rm Aut}_{\mathfrak C}(\mathcal O)$ of an object $\mathcal O$ within $\mathfrak C$ is then identified with the centraliser $C:=C_S(G)$ in $S$ of the monodromy group $G:=\theta(\Gamma)$ of $\mathcal O$. Now $\mathcal O$ is connected if and only if $G$ is transitive on $\Omega$, as we will assume until the final section of this paper. In this situation, an important result is the following:

\begin{theo}\label{isothm}
If $\mathcal O$ is connected then
\[{\rm Aut}_{\mathfrak C}({\mathcal O})\cong N_G(H)/H\cong N_{\Gamma}(M)/M,\]
where $H$ and $M$ are the stabilisers in $G$ and $\Gamma$ of some $\alpha\in\Omega$. 
\end{theo}

There are analogous results in other contexts, ranging from abstract polytopes to covering spaces. Proofs of Theorem~\ref{isothm} for particular categories can be found in the literature: for instance, in~\cite{JS} it is deduced for oriented maps from a more general result about morphisms in that category; in~\cite[Theorem~2.2 and Corollary~2.1]{JW} a proof for dessins is briefly outlined; analogous results for covering spaces are proved in~\cite[Appendix]{Mas} and~\cite[Theorem~81.2]{Mun}, and for abstract polytopes in~\cite[Propositions~2D8 and 2E23(a)]{MS}. In \S\ref{autoproof} we give a detailed proof of
the following `folklore' theorem about permutation groups, which is applicable to a wide class of categories; these are the permutational categories defined in~\cite{Jon16} and in~\S\ref{permcats}, including all the categories mentioned above.

\begin{theo}\label{autothm}
Let $G$ be a transitive finite permutation group on a set $\Omega$, with $H$ the stabiliser of some $\alpha\in\Omega$, and let $C:=C_S(G)$ be the centraliser of $G$ is the symmetric group $S:={\rm Sym}(\Omega)$. Then
\begin{enumerate}
\item $C \cong N_G(H)/H$,
\item $C$ acts regularly on the set $\Phi$ of elements of $\Omega$ with stabiliser $H$.
\end{enumerate}
\end{theo}

One sometimes finds proofs or statements of particular cases of Theorem~\ref{autothm} which include the assertion that $C$ acts regularly on the set $\Phi$ of fixed points of $H$ in $\Omega$; while this is valid if $H$ is finite, in \S\ref{alternative} we give counter-examples, based on Baumslag-Solitar groups~\cite{BS}, to show that if $H$ is infinite then $\Phi$ must be redefined more precisely as
in (2). It follows from Theorem~\ref{isothm} that if the monodromy group $G$ of an object $\mathcal O$ acts primitively on $\Omega$, then either ${\rm Aut}_{\mathfrak C}(\mathcal O)$ is trivial, or $G$ is a cyclic group of prime order, acting regularly on $\Omega$; in \S\ref{primitive} we describe the objects with the latter property in various categories $\mathfrak C$. In \S\ref{noncon} we briefly consider the structure and cardinality of the automorphism groups of non-connected objects in permutational categories.

\medskip

\noindent{\bf Acknowledgments.} The author is grateful to Ernesto Girondo, Gabino Gonz\'alez-Diez and Rub\'en Hidalgo for discussions about dessins d'enfants which motivated this work.


\section{Permutational categories}\label{permcats}

A {\em permutational category\/} $\mathfrak C$ is defined in~\cite{Jon16} to be a category in which the objects $\mathcal O$ can be identified with the permutation representations $\theta:\Gamma\to S:={\rm Sym}(\Omega)$ of a {\em parent group\/} $\Gamma=\Gamma_{\mathfrak C}$, and the morphisms ${\mathcal O}_1\to{\mathcal O}_2$ with the $\Gamma$-invariant functions $\Omega_1\to\Omega_2$, those commuting with the actions of $\Gamma$ on $\Omega_1$ and $\Omega_2$. The {\em automorphism group\/} ${\rm Aut}({\mathcal O})={\rm Aut}_{\mathfrak C}({\mathcal O})$ of $\mathcal O$ in the category $\mathfrak C$ is then the group of all permutations of $\Omega$ commuting with $\Gamma$, that is, the centraliser $C_S(G)$ of the {\em monodromy group\/} $G=\theta(\Gamma)$ of $\mathcal O$ in the symmetric group $S$. Here we will restrict our attention to the {\em connected\/} objects, those for which $\Gamma$ acts transitively on $\Omega$.

We will concentrate mainly on five particular examples of permutational categories, outlined briefly below (for full details, and other examples, see~\cite{Jon16}). In each case, the parent group $\Gamma$ is either an extended triangle group
\[\Delta[p,q,r]=\langle R_0, R_1, R_2\mid R_i^2=(R_1R_2)^p=(R_2R_0)^q=(R_0R_1)^r=1\rangle,\]
or its orientation-preserving subgroup of index $2$, the triangle group
\[\Delta(p,q,r)=\langle X, Y, Z\mid X^p=Y^q=Z^r=XYZ=1\rangle,\]
where $X=R_1R_2$, $Y=R_2R_0$ and $Z=R_0R_1$. Here $p, q, r\in{\mathbb N}\cup\{\infty\}$, and we ignore any relations of the form $W^{\infty}=1$. In what follows, $*$ denotes a free product, $C_n$ denotes a cyclic group of order $n\in{\mathbb N}\cup\{\infty\}$, while $V_n$ is an elementary abelian group of order $n$, and $F_r$ is a free group of rank $r$.

\smallskip

\noindent{\bf 1.} The category $\mathfrak M$ of maps on surfaces (possibly non-orientable or with boundary), that is, embeddings of graphs with simply connected faces, has parent group
\[\Gamma=\Gamma_{\mathfrak M}=\Delta[\infty, 2, \infty]\cong V_4*C_2.\]
This permutes the set $\Omega$ of incident vertex-edge-face flags of a map (equivalently, the faces of its barycentric subdivision), with each involution $R_i\;(i=0, 1, 2)$ changing the $i$-dimensional component of each flag (whenever possible) while preserving the other two.

\smallskip

\noindent{\bf 2.} The category $\mathfrak M^+$ of oriented maps, those in which the underlying surface is oriented and without boundary,  has parent group
\[\Gamma=\Gamma_{\mathfrak M^+}=\Delta(\infty, 2, \infty)\cong C_{\infty}*C_2.\]
This group permutes directed edges, with $X$ usng the local orientation to rotate them about their target vertices, and $Y$ reversing their direction, so that $Z$ rotates them around incident faces.

\smallskip

\noindent{\bf 3.} There are several ways of defining or representing hypermaps. For our purposes, the most convenient is the Walsh representation as a bipartite map~\cite{Wal}, in which the white and black vertices of the embedded graph correspond to the hypervertices and hyperedges of the hypermap, the edges correspond to incidences between them, and the faces correspond to its hyperfaces. The category $\mathfrak H$ of all hypermaps, where the underlying surface is unoriented and possibly with boundary, has parent group
\[\Gamma=\Gamma_{\mathfrak H}=\Delta[\infty, \infty, \infty]\cong C_2*C_2*C_2.\]
This permutes incident edge-face pairs of the bipartite map, with involutions $R_0$ and $R_1$ preserving the face and the incident black and white vertex respectively, while $R_2$ preserves the edge.

\smallskip

\noindent{\bf 4.}  The category $\mathfrak H^+$ of oriented hypermaps, those in which the underlying surface is oriented and without boundary, has parent group
\[\Gamma=\Gamma_{\mathfrak H^+}=\Delta(\infty, \infty, \infty)\cong C_{\infty}*C_{\infty}\cong F_2.\]
This permutes the edges of the embedded graph, with $X$ and $Y$ using the local orientation to rotate them around their incident white and black vertices, so that $Z$ rotates them around incident faces.

\smallskip

\noindent{\bf 5.} The category $\mathfrak D$ of dessins d'enfants can be identified with the subcategory of $\mathfrak H^+$ consisting of its finite objects, those in which the embedded bipartite graph is finite and the surface is compact. It has the same parent group $\Gamma=\Delta(\infty, \infty, \infty)\cong F_2$ as $\mathcal H^+$, permuting edges as before.

\smallskip

 If we wish to restrict any of these categories to the subcategory of objects of a particular type $(p,q,r)$, we replace the parent group given above with the corresponding triangle or extended triangle group of that type.

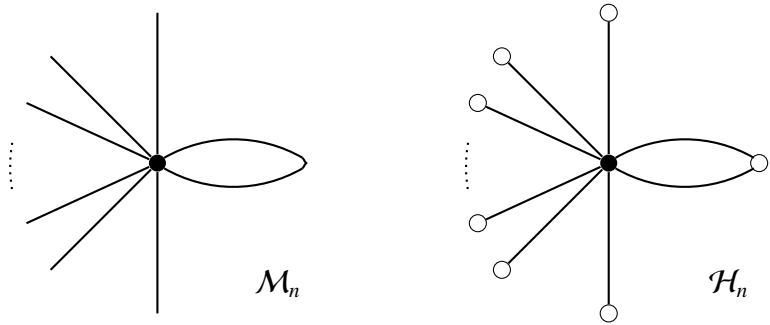
\begin{figure}[h!]
\begin{center}
\begin{tikzpicture}[scale=0.2, inner sep=0.8mm]

\node (A) at (0,0) [shape=circle, fill=black] {};
\draw [thick] (9.65,0.35) arc (60:120:9.2);
\draw [thick] (9.65,-0.35) arc (-60:-120:9.2);
\draw [thick] (9.65,0.35) to (9.75,0.2) to (9.8,0.1) to (9.9,0) to (9.8,-0.1) to (9.75,-0.2) to (9.65,-0.35);
\draw [thick] (0,10) to (A) to (0,-10);
\draw [thick] (-7.1,7.1) to (A) to (-7.1,-7.1);
\draw [thick] (-8.7,4) to (A) to (-8.7,-4);
\draw [thick, dotted] (-9.65,1.35) arc (170:190:9.2);

\node at (8,-8) {${\mathcal M}_n$};


\node (A') at (30,0) [shape=circle, fill=black] {};
\node (0') at (40,0) [shape=circle, draw] {};
\node (1') at (30,10) [shape=circle, draw] {};
\node (-1') at (30,-10) [shape=circle, draw] {};
\draw [thick] (39.65,0.35) arc (60:120:9.2);
\draw [thick] (39.65,-0.35) arc (-60:-120:9.2);
\node (2') at (22.9,7.1) [shape=circle, draw] {};
\node (-2') at (22.9,-7.1) [shape=circle, draw] {};
\node (3') at (21.3,4) [shape=circle, draw] {};
\node (-3') at (21.3,-4) [shape=circle, draw] {};
\draw [thick] (1') to (A') to (-1');
\draw [thick] (2') to (A') to (-2');
\draw [thick] (3') to (A') to (-3');
\draw [thick, dotted] (20.65,1.35) arc (170:190:9.2);

\node at (38,-8) {${\mathcal H}_n$};

\end{tikzpicture}
\end{center}
\caption{The map ${\mathcal M}_n$ and the hypermap ${\mathcal H}_n$} 
\label{MnHn}
\end{figure}

 \medskip

\begin{exam}
The planar map ${\mathcal M}_n\in{\mathfrak M}^+\;(n\ge 2)$, shown on the left in Figure~\ref{MnHn}, has one vertex, of valency $n$, incident with one loop and $n-2$ half edges. It corresponds to the epimorphism
 \[\theta:\Delta(\infty,2,\infty)\to S_n,\quad X\mapsto x=(1,2,\ldots, n),\; Y\mapsto y=(1,2),\; Z\mapsto z=(n, n-1, \ldots, 2).\]
This map can be regarded as a hypermap ${\mathcal H}_n$, shown on the right in Figure~\ref{MnHn}, by adding a white vertex to each edge; this corresponds to composing $\theta$ with the natural epimorphism $\Delta(\infty,\infty,\infty)\to\Delta(\infty,2,\infty)$. The hypermap ${\mathcal H}_n$ has type $(n,2,n-1)$, and it can be regarded as a member of the category of oriented hypermaps of this type by factoring $\theta$ through $\Delta(n,2,n-1)$. In all cases the monodromy group $G$ is $S_n$, in its natural representation, and the automorphism group is trivial. However, if we regard ${\mathcal M}_n$ or ${\mathcal H}_n$ as an unoriented map or hypermap, then its monodromy group is a permutation group of degree $2n$, and the automorphism group has order $2$, generated by the obvious reflection.
\end{exam}

\medskip

We briefly mention two other classes of permutational category in which Theorem~\ref{isothm} applies. The first concerns abstract polytopes~\cite{MS}, which can be regarded as higher-dimensional generalisations of maps. Those $n$-polytopes of a particular type, associated with the Schl\"afli symbol $\{p_1, \ldots, p_{n-1}\}$, can be regarded as transitive permutation representations of the Coxeter group $\Gamma$ with presentation
\[\langle R_0, \ldots, R_n\mid R_i^2=(R_{i-1}R_i)^{p_i}=(R_iR_j)^2=1\;(|i-j|>1)\rangle.\]
For instance, $\Gamma_{\mathfrak M}$ is associated with the Schl\"afli symbol $\{\infty,\infty\}$. However, in higher dimensions, not all transitive representations correspond to abstract polytopes, since the monodromy groups must satisfy the intersection property~\cite[Proposition~2B10]{MS}.

The second class of examples concerns covering spaces~\cite{Mas, Mun}. Under suitable connectedness assumptions, the (connected, unbranched) coverings $Y\to X$ of a topological space $X$ can be identified with the transitive permutation representations $\theta:\Gamma\to S={\rm Sym}(\Omega)$ of its fundamental group $\Gamma=\pi_1X$, acting by unique path-lifting on the fibre $\Omega$ over a base-point in $X$. In this case the automorphism group of an object $Y\to X$ in this category is the group of covering transformations, the centraliser in $S$ of the monodromy group $\theta(\Gamma)$ of the covering.


This last example helps to explain the importance of the fifth category listed above, the category $\mathfrak D$ of dessins d'enfants.
By Belyi's Theorem~\cite{Bel}, a compact Riemann surface $R$ is defined (as a projective algebraic curve) over the field $\overline{\mathbb Q}$ of algebraic numbers if and only if it admits a non-constant meromorphic function $\beta$ branched over at most three points of the complex projective line (the Riemann sphere) ${\mathbb P}^1({\mathbb C})=\hat{\mathbb C}={\mathbb C}\cup\{\infty\}$. Composing $\beta$ with a M\"obius transformation if necessary, we may assume that its critical values are contained in $\{0, 1, \infty\}$. Such {\em Belyi functions\/} $\beta$ correspond to unbranched finite coverings $R\setminus\beta^{-1}(\{0, 1, \infty\})\to X$ of the thrice-punctured sphere $X=\hat{\mathbb C}\setminus\{0, 1, \infty\}$, and hence to transitive finite permutation representations of its fundamental group $\Gamma=\pi_1X$; this is a free group of rank $2$, freely generated by the homotopy classes of small loops around $0$ and $1$. The unit interval $[0,1]\subset \hat{\mathbb C}$ lifts to a bipartite graph embedded in $R$, with white and black vertices over $0$ and $1$, and face-centres over $\infty$. Conversely, any finite oriented hypermap, after suitable uniformisation, yields a Riemann surface $R$ defined over $\overline{\mathbb Q}$; see~\cite{GG, Gro, JW, LZ} for details of these connections, and of the action of the absolute Galois group ${\rm Gal}(\overline{\mathbb Q}/{\mathbb Q})$ on dessins.



\section{Proof of Theorems~\ref{isothm} and~\ref{autothm}}\label{autoproof}

Let $\mathcal O$ be an object in a permutational category $\mathfrak C$, identified with a transitive permutation representation $\Gamma\to G\le S:={\rm Sym}(\Omega)$, so that its automorphism group ${\rm Aut}_{\mathfrak C}({\mathcal O})$ is identified with the centraliser $C_{S}(G)$ of $G$ in $S$. Then Theorem~\ref{isothm} asserts that 
\[{\rm Aut}_{\mathfrak C}({\mathcal O}) \cong N_G(H)/H\cong N_{\Gamma}(M)/M,\]
where $G$ is the monodromy group $\mathcal O$, and $H$ and $M$ are point-stabilisers in $G$ and $\Gamma$. The second isomorphism follows immediately from the first, and this in turn follows from part (1) of Theorem~\ref{autothm}, both parts of which we will now prove.

\begin{proof} The centraliser $C$ of $G$ acts semi-regularly (i.e.~freely) on $\Omega$. To see this, suppose that $c\in C$ and $\beta c=\beta$ for some $\beta\in\Omega$. Given any $\omega\in\Omega$, there is some $g\in G$ such that $\omega = \beta g$, since $G$ is transitive on $\Omega$. Then $\omega c = (\beta g)c = (\beta c)g = \beta g = \omega$. Thus $c = 1$, as required.

Let $\Phi=\{\beta\in\Omega \mid G_{\beta}=H\}$, so in particular $\alpha\in\Phi$. Then $C$ leaves $\Phi$ invariant, since if $\beta\in\Phi$ and $c\in C$ then for all $h\in H$ we have $(\beta c)h = (\beta h)c = \beta c$, so that $\beta c\in\Phi$.

Let us identify $\Omega$ with the set of cosets of $H$ in $G$ in the usual way, identifying each $\omega\in\Omega$ with the unique coset $Hx$ such that $x\in G$ and $\alpha x = \omega$. Thus $\alpha$ is identified with $H$ itself, and $G$ acts on the cosets by $g : Hx\mapsto Hxg$. 

Then $\Phi$ is identified with the set of cosets of $H$ in $N:=N_G(H)$. To see this, let $\omega\in\Omega$ correspond to a coset $Hx$ of $H$ in $G$. First suppose that $x=n\in N$. Then $(Hn)h=(nH)h=n(Hh)= nH=Hn$ for all $h\in H$, so the coset $Hn$ is fixed by $H$, giving $H\le G_{\omega}$, while if $g\in G_{\omega}$ then $Hng=Hn$, so $H=H^n=H^ng=Hg$ and hence $g\in H$, giving $G_{\omega}\le H$. Thus $G_{\omega}=H$ and hence $\omega\in\Phi$. Conversely, suppose that $\omega\in\Phi$. Then $G_{\omega}=H$, so $Hxg=Hx$ if and only if $g\in H$, that is, $H^xg=H^x$ if and only if $g\in H$, so $H^x=H$ and hence $x\in N$.

Let us define a new action of $N$ on $\Omega$ (now regarded as the set of cosets of $H$ in $G$) by $n : Hx\mapsto n^{-1}Hx=Hn^{-1}x$ for all $n\in N$ and $x\in G$. If $n_1, n_2\in N$ then $n_1$, followed by $n_2$, sends $Hx$ to $n_2^{-1}n_1^{-1}Hx = (n_1n_2)^{-1}Hx $, as does $n_1n_2$, so this is indeed a group action of $N$. It commutes with the action of $G$ on $\Omega$, since $n^{-1}(Hxg)=(n^{-1}Hx)g$ for all $n\in N$ and $x, g\in G$, so it defines a homomorphism $\theta:N\to C$. In particular, this action preserves $\Phi$, since $C$ does.

The induced action of $N$ on $\Phi$ is transitive, since if $n', n''\in N$ then the element $n=n'(n'')^{-1}\in N$ sends $Hn'$ to $n^{-1}Hn'=Hn^{-1}n'=Hn''$.  Thus $\theta(N)$ acts on $\Phi$ as a transitive subgroup of $C$. But $C$ acts semi-regularly on $\Phi$, so it has no transitive proper subgroups. Therefore $\theta$ is an epimorphism, and $C$ acts regularly on $\Phi$, giving (2).

In this action of $N$, we have $n^{-1}Hx=Hx$ if and only if $n\in H$, so the subgroup stabilising any coset $Hx$ is $H$, which is therefore the kernel $\ker(\theta)$ of this action of $N$. The First Isomorphism Theorem therefore gives $N/H\cong C$, so (1) is proved.
\end{proof}

\begin{rema}
The most symmetric objects in $\mathfrak C$ are the {\em regular\/} objects, those for which ${\rm Aut}_{\mathfrak C}({\mathcal O})$ acts transitively on $\Omega$. By Theorem~\ref{autothm}(2) this is equivalent to $\Phi=\Omega$, that is, to $H=1$, meaning that $G$ acts regularly on $\Omega$. This is also equivalent to $M$ being a normal subgroup of $\Gamma$. Then $G\cong C\cong \Gamma/M$, and $G$ and $C$ can be identified with the right and left regular representations of the same group. (In fact, if $G$ is abelian then $C=G$.)
\end{rema}


\section{An alternative form of Theorem~\ref{autothm}(2)}\label{alternative}

One sometimes finds part (2) of Theorem~\ref{autothm} stated in the following alternative form:

\smallskip

\indent\indent $(2')$ $C$ acts regularly on the set of fixed points of $H$ in $\Omega$.

\smallskip

\noindent This is equivalent to (2) in cases where $\Omega$ is finite, or more generally where $H$ is finite, so that an inclusion $H=H_{\alpha}\le H_{\beta}$ between conjugate subgroups is equivalent to their equality. However, $(2')$ can be false if $H$ is infinite, as shown by the following example.

\begin{exam}
Let $G$ be the Baumslag-Solitar group~\cite{BS}
\[G=BS(1,2)=\langle a, b\mid a^b=a^2\rangle.\]
This is a semidirect product $G = A\rtimes B$,
where $B=\langle b\rangle\cong C_{\infty}$,
and $A$ is the normal closure of $a$ in $G$, an abelian group of countably infinite rank, generated by the conjugates
\[a_i:=a^{2^i}=a^{b^i}\quad (i\in{\mathbb Z})\]
of $a$, with $a_i^2=a_{i+1}$ for all $i\in{\mathbb Z}$. This subgroup $A$ can be identified with the additive group of the ring ${\mathbb Z}[\frac{1}{2}]$, with each element $\prod_ia_i^{e_i}$ corresponding to $\sum_ie_i2^i\in{\mathbb Z}[\frac{1}{2}]$.
Thus $a=a_0$ corresponds to the element $1\in{\mathbb Z}[\frac{1}{2}]$, and $b$, acting by conjugation on $A$ as the automorphism $a_i\mapsto a_{i+1}$, acts on ${\mathbb Z}[\frac{1}{2}]$ by $t\mapsto 2t$. In particular, the subgroup $H:=\langle a\rangle$ has conjugate subgroups $H_i:=H^{b^{-i}}=\langle a^{2^{-i}}\rangle$ for all $i\in{\mathbb Z}$, with a chain of index $2$ inclusions
\[\cdots <H_{-2}<H_{-1}<H_0\;(=H)<H_1<H_2<\cdots.\]


In order to represent $G$ on the cosets of this subgroup $H$, first let $G$ act on the disjoint union
\[\Omega=\bigcup_{i\in{\mathbb Z}}\Omega_i\]
of countably many copies $\Omega_i$ of the Sylow $2$-subgroup
\[{\mathbb T}:={\rm Syl}_2({\mathbb Q}/{\mathbb Z})={\mathbb Z}[\hbox{$\frac{1}{2}$}]/{\mathbb Z}\cong C_{2^{\infty}}:=\bigcup_{e\ge 0}C_{2^e}\]
of ${\mathbb Q}/{\mathbb Z}$ as follows. Let the generator $a$ of $G$ act on each $\Omega_i$ by
\[a: t\mapsto t+2^i \; {\rm mod}\,{\mathbb Z}\]
for $t\in{\mathbb Z}[\frac{1}{2}]$, so in particular $a$ fixes $\Omega_i$ for each $i\ge 0$, and has cycles of length $2^{-i}$ on $\Omega_i$ for each $i<0$. Let the generator $b$ act on $\Omega$ by a translation $i\mapsto i-1$ of subscripts, so that it sends each element $t_i=t\in{\mathbb T}$ of $\Omega_i$ to the corresponding element $t_{i-1}=t$ of $\Omega_{i-1}$. Then the element $a^b=b^{-1}ab$ acts on each $\Omega_i$ by a composition of three permutations
\[a^b : t=t_i\mapsto t_{i+1}\mapsto t_{i+1}+2^{i+1}\mapsto t_i+2^{i+1} =  t+2^{i+1}\; {\rm mod}\,{\mathbb Z},\]
which has the same effect as
\[a^2 : t\mapsto t+2^{i+1}\; {\rm mod}\,{\mathbb Z},\]
so we have a group action of $G$ on $\Omega$.

It is easy to see that $G$ acts faithfully and transitively on $\Omega$, that the stabiliser of the element $\alpha=0_0\in\Omega_0$ is $H=\langle a\rangle$, and that $N_G(H)=A$, so that
\[N_G(H)/H=A/H\cong{\mathbb T}.\]

We now calculate
\[C:=C_S(G)=C_S(A)\cap C_S(B),\]
where $S:={\rm Sym}(\Omega)$. First note that $C_S(A)$ must permute the orbits of $A$, which are the sets $\Omega_i$, and must do so trivially since $A$ has a different representation, with kernel $H_i=\langle a^{2^{-i}}\rangle$, on each $\Omega_i$. Since $A$ induces the regular representation of the abelian group $A/H_i$ on $\Omega_i$, it follows that $C_S(A)$ must be the cartesian product $\prod_{i\in{\mathbb Z}}A/H_i$ of the groups $A/H_i$, each factor $A/H_i$ acting regularly on $\Omega_i$ and fixing $\Omega\setminus\Omega_i$. Even though the subgroups $H_i$ are all distinct, we have $A/H_i\cong{\mathbb T}$ for all $i\in{\mathbb Z}$, so $C_S(A) \cong{\mathbb T}^{\mathbb Z}$. The only elements of $C_S(A)$ commuting with the subscript shift $b$ are those corresponding to elements of the diagonal subgroup of ${\mathbb T}^{\mathbb Z}$, inducing the same permutation on each subset $\Omega_i$. These form a group isomorphic to $\mathbb T$, proving that
\[C\cong N_G(H)/H.\]

Thus $G$ satisfies Theorem~\ref{autothm}(1), but what about statements (2) and $(2')$? In this example, the orbits of $C$ are the sets $\Omega_i$, each permuted regularly by $C$, while the subset of $\Omega$ fixed by $H$ is the disjoint union
\[\bigcup_{i\ge 0}\Omega_i\]
of infinitely many of these orbits. Thus $G$ does not satisfy statement $(2')$, though it does satisfy $(2)$ since the points with stabiliser $H$ are those in $\Omega_0$, forming a regular orbit of $C$; the points in $\Omega_i$ for $i>0$, although they are also fixed by $H$, have stabiliser $H_i$ properly containing $H$.
\end{exam}

\begin{rema}
In this example, one can regard $G$ as the monodromy group of an infinite oriented hypermap $\mathcal H$, given by the epimorphism $\theta: F_2\to G, X\mapsto a, Y\mapsto b$. Then ${\rm Aut}_{\mathfrak H^+}({\mathcal H})\cong C\cong{\mathbb T}$.
\end{rema}

\begin{rema}
There is an obvious generalisation of this example to the Baumslag-Solitar group $G=BS(1,q)=\langle a, b\mid a^b=a^q\rangle$ for an arbitrary integer $q\ne 0, \pm 1$. See~\cite{Jon08} for a discussion of the oriented hypermaps associated with the Baumslag-Solitar groups $BS(p,q)=\langle a, b\mid (a^p)^b=a^q\rangle$ for arbitrary $p, q\ne 0$.
\end{rema}


\section{Primitive monodromy groups}\label{primitive}

If $G$ is a permutation group on a set $\Omega$, then the relation $G_{\alpha}=G_{\beta}$, appearing in Theorem~\ref{autothm} via the definition of $\Phi$, is a $G$-invariant equivalence relation on $\Omega$, and its equivalence classes are the orbits of the centraliser $C$ of $G$. Recall that a permutation group is {\em primitive\/} if it preserves no non-trivial equivalence relation; equivalently, the point-stabilisers are maximal subgroups. As an immediate consequence of Theorem~\ref{autothm}, we have:

\begin{cor}
If $G$ is a primitive permutation group, then either $G\cong C_p$, acting regularly, for some prime $p$, with centraliser $C=G$, or the centraliser $C$ of $G$ is the trivial group.
\end{cor}

\begin{proof}
The equivalence relation $G_{\alpha}=G_{\beta}$ on $\Omega$ must be either the identity or the universal relation. In the first case the equivalence classes are singletons, so $|C|=1$. In the second case $G_{\alpha}=\{1\}$; this is a maximal subgroup of $G$, so $G\cong C_p$ for some prime $p$, with $C=G$.
\end{proof}

\begin{cor}
In a permutational category $\mathfrak C$, if the monodromy group $G$ of an object $\mathcal O$ is a primitive permutation group, then either $\mathcal O$ is regular, with ${\rm Aut}_{\mathfrak C}({\mathcal O})=G\cong C_p$ for some prime $p$, or ${\rm Aut}_{\mathfrak C}({\mathcal O})$ is the trivial group.
\end{cor}

Of course, there are many examples of primitive permutation groups, both sporadic and members of infinite families: just represent a group on the cosets of a maximal subgroup. On the other hand, Cameron, Neumann and Teague~\cite{CNT} have shown that, for a set of integers $n$ of asymptotic density $1$, the only primitive groups of degree $n$ are the alternating and symmetric groups $A_n$ and $S_n$.

\begin{exam}
 The symmetric and alternating groups, in their natural representationas, arise quite frequently as monodromy groups in various categories. For instance, let $\mathfrak C=\mathfrak H^+$, the category of oriented hypermaps, with parent group $\Gamma=F_2=\langle X, Y\mid -\rangle$. A theorem of Dixon~\cite{Dix} states that a randomly chosen pair of permutations $x, y\in S_n$ generate $S_n$ or $A_n$ with probability approaching $3/4$ or $1/4$ as $n\to\infty$, so in that sense `most' finite objects in this category have a symmetric or alternating monodromy group, and a trivial automorphism group.
\end{exam}

In most of the permutational categories of current interest, it is simple to describe the regular objects with automorphism group $C_p$ for each prime $p$; apart from these exceptions, objects with a primitive monodromy group have a trivial automorphism group. The exceptions correspond to the normal subgroups of index $p$ in the parent group $\Gamma$, or equivalently to the maximal subgroups in the elementary abelian $p$-group $\Gamma/\Gamma'\Gamma^p$, where $\Gamma'$ and $\Gamma^P$ are the subgroups of $\Gamma$ generated by the commutators and $p$-th powers. In the categories listed earlier, they are as follows.

If ${\mathfrak C}={\mathfrak H}^+$ or $\mathfrak D$ then $\Gamma=F_2$, so $\Gamma/\Gamma'\Gamma^p\cong C_p\times C_p$, with $p+1$ maximal subgroups. Of the corresponding oriented hypermaps, three have type a permutation of $(p,p,1)$ and are planar, while the remaining $p-2$ have type $(p,p,p)$ and  genus $(p-1)/2$. As dessins, the former are on the Riemann sphere, with Belyi functions $\beta: z\mapsto z^p$, $1/(1-z^p)$ and $1-z^{-p}$, and automorphisms $z\mapsto \zeta z$ where $\zeta^p=1$. The latter are on Lefschetz curves $y^p=x^u(x-1)$ for $u=1,\ldots, p-2$, each with a Belyi function $\beta:(x,y)\mapsto x$ and automorphisms $(x,y)\mapsto (x,\zeta y)$ where $\zeta^p=1$ (see~\cite[Example 5.6]{JW}). The four dessins for $p=3$ are shown in Figure~\ref{p=3}; in the dessin on the right, opposite sides of the hexagon are identified to form a torus.

\begin{figure}[h!]
\begin{center}
\begin{tikzpicture}[scale=0.2, inner sep=0.8mm]

\node (a) at (-25,-4) [shape=circle, draw] {};
\node (b) at (-25,4) [shape=circle, fill=black] {};
\draw [thick] (a) to (b);
\draw [thick] (-24.8,-3.6) arc (-45:45:5.2);
\draw [thick] (-25.2,3.6) arc (135:223:5.2);


\node (c) at (-10,0) [shape=circle, draw] {};
\node (d) at (-5,0) [shape=circle, fill=black] {};
\node (e) at (-12.5,4) [shape=circle, fill=black] {};
\node (f) at (-12.5,-4) [shape=circle, fill=black] {};
\draw [thick] (c) to (d);
\draw [thick] (c) to (e);
\draw [thick] (c) to (f);


\node (g) at (10,0) [shape=circle, fill=black] {};
\node (h) at (15,0) [shape=circle, draw] {};
\node (i) at (7.5,4) [shape=circle, draw] {};
\node (j) at (7.5,-4) [shape=circle, draw] {};
\draw [thick] (g) to (h);
\draw [thick] (g) to (i);
\draw [thick] (g) to (j);


\node (k) at (35,0) [shape=circle, draw] {};
\node (l) at (32.5,4) [shape=circle, fill=black] {};
\node (m) at (27.5,4) [shape=circle, draw] {};
\node (n) at (25,0) [shape=circle, fill=black] {};
\node (o) at (27.5,-4) [shape=circle, draw] {};
\node (p) at (32.5,-4) [shape=circle, fill=black] {};
\draw [thick] (k) to (l) to (m) to (n) to (o) to (p) to (k);

\end{tikzpicture}
\end{center}
\caption{The four dessins with primitive monodromy group $C_p$, $p=3$} 
\label{p=3}
\end{figure}
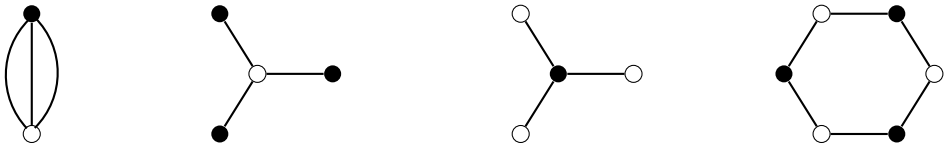

If ${\mathfrak C}={\mathfrak M}^+$ then $\Gamma=C_{\infty}*C_2$, so $\Gamma/\Gamma'\Gamma^p\cong V_4$ or $C_p$ as $p=2$ or $p>2$, giving three oriented maps or one, all planar. Their types are the three permutations of $(2,2,1)$, together with $(p,1,p)$. They are shown, for $p=2$ and $3$, in Figure~\ref{p=2,3}.

\begin{figure}[h!]
\begin{center}
\begin{tikzpicture}[scale=0.2, inner sep=0.8mm]

\node (a) at (-25,0) [shape=circle, fill=black] {};
\draw [thick] (-25,0) arc (0:360:3);


\node (b) at (-10,0) [shape=circle, fill=black] {};
\draw [thick] (-15,0) to (b) to (-5,0);


\node (c) at (5,0) [shape=circle, fill=black] {};
\node (d) at (15,0) [shape=circle, fill=black] {};
\draw [thick] (c) to (d);


\node (c) at (27.5,0) [shape=circle, fill=black] {};
\draw [thick] (c) to (32.5,0);
\draw [thick] (c) to (25,4);
\draw [thick] (c) to (25,-4);

\end{tikzpicture}
\end{center}
\caption{The four oriented maps with primitive monodromy group $C_p$, $p=2, 3$} 
\label{p=2,3}
\end{figure}
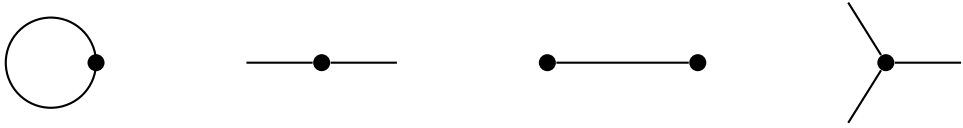

\begin{figure}[h!]
\begin{center}
\begin{tikzpicture}[scale=0.2, inner sep=0.8mm]

\node (a) at (-25,0) [shape=circle, draw] {};
\node (b) at (-20,0) [shape=circle, fill=black] {};
\draw [thick] (a) to (b);


\draw [thick, dashed] (0,5.4) arc (10:350:3);
\node (c) at (0,5) [shape=circle, draw] {};
\node (d) at (-6,5) [shape=circle, fill=black] {};
\draw [thick] (c) to (d);


\node (e) at (0,-5) [shape=circle, draw] {};
\node (f) at (-6,-5) [shape=circle, fill=black] {};
\draw [thick] (-0.2,-4.5) arc (10:350:2.8);


\draw [thick, dashed] (15,5.4) arc (10:350:3);
\node (g) at (15,5) [shape=circle, draw] {};
\node (h) at (12,5) [shape=circle, fill=black] {};
\draw [thick] (g) to (h);


\draw [thick, dashed] (15,-4.6) arc (10:270:3);
\node (i) at (15,-5) [shape=circle, draw] {};
\node (j) at (12,-2) [shape=circle, fill=black] {};
\node (k) at (12,-8) [shape=circle, fill=black] {};
\draw [thick] (15,-4.6) arc (10:90:3);
\draw [thick] (12,-8) arc (-90:-10:3);


\draw [thick, dashed] (30,5) arc (0:360:3);
\node (l) at (30,5) [shape=circle, fill=black] {};
\node (m) at (27,5) [shape=circle, draw] {};
\draw [thick] (l) to (m);


\draw [thick, dashed] (26.5,-2.1) arc (95:265:3);
\node (n) at (30,-5) [shape=circle, fill=black] {};
\node (o) at (27,-2) [shape=circle, draw] {};
\node (p) at (27,-8) [shape=circle, draw] {};
\draw [thick] (30,-4.6) arc (10:83:3);
\draw [thick] (27.5,-8) arc (-83:-10:3);

\end{tikzpicture}
\end{center}
\caption{The seven hypermaps with primitive monodromy group $C_2$} 
\label{7hypermaps}
\end{figure}
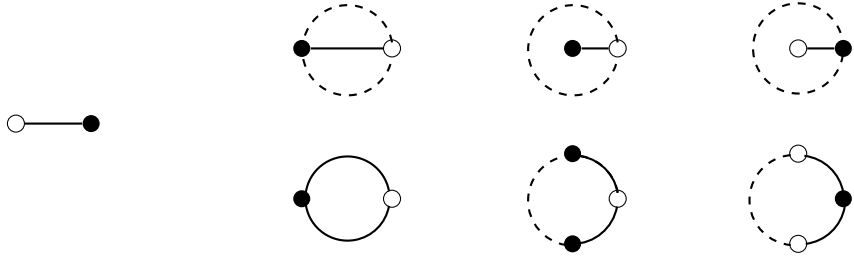

If ${\mathfrak C}={\mathfrak H}$ or $\mathfrak M$ then $\Gamma=C_2*C_2*C_2$ or $V_4*C_2$, so in either case $\Gamma/\Gamma'\Gamma^p\cong V_8$ or $1$ as $p=2$ or $p>2$. If $p=2$ there are seven hypermaps and seven maps; if $p>2$ there are none. The hypermaps are shown in Figure~\ref{7hypermaps}. The hypermap on the left is planar, of type $(2,2,2)$. The other six are on the closed disc, shown by a broken line; those in the first, second and third columns have type $(2,2,1)$, $(2,1,2)$ and $(1,2,2)$ respectively. The seven maps can be obtained from these hypermaps by ignoring all vertices of one particular colour.

If $X$ is a compact orientable surface of genus $g$ then there are $(p^{2g}-1)/(p-1)$ regular coverings $Y\to X$ with monodromy group $C_p$, corresponding to the normal subgroups of index $p$ in $\pi_1X=\langle A_i, B_i\;(i=1,\ldots, g)\mid \prod_i[A_i,B_i]=1\rangle$. In the non-orientable case, with $\pi_1X=\langle R_i\;(i=1,\ldots, g)\mid \prod_i R_i^2=1\rangle$ there are $(p^{g-1}-1)/(p-1)$ or $2^g-1$ such coverings as $p>2$ or $p=2$.


\section{Automorphism groups of non-connected objects}\label{noncon}

In most applications, it is traditional and convenient to restrict attention to the connected objects in a category, as we have done so far. Here we will briefly show how Theorem~\ref{isothm} extends to objects which are not connected.

If $\mathcal C$ is a permutational category with parent group $\Gamma$, then the connected components ${\mathcal O}_i\;(i\in I)$ of an object ${\mathcal O}$ in $\mathfrak C$ correspond bijectively to the orbits $\Omega_i\;(i\in I)$ of $\Gamma$ on the set $\Omega$ associated with $\mathcal O$. As before, ${\rm Aut}_{\mathfrak C}({\mathcal O)}$ is isomorphic to the centraliser $C$ of $\Gamma$ in $S={\rm Sym}(\Omega)$. In order to describe the structure of $C$ in general, we first consider two extreme cases.

Suppose first that the components $\mathcal O_i$ are mutually non-isomorphic. This is equivalent to the point stabilisers $\Delta_i\le\Gamma$ for different orbits $\Omega_i$ being mutually non-conjugate in $\Gamma$. Then $C$ is the cartesian product of the centralisers $C_i\le{\rm Sym}(\Omega_i)$ of $\Gamma$ on the sets $\Omega_i$. By the transitivity of $\Gamma$ on $\Omega_i$, we have $C_i\cong N_{G_i}(H_i)\cong N_{\Gamma}(\Delta_i)/\Delta_i$ for each $i\in I$, where $G_i$ is the permutation group induced by $\Gamma$ on $\Omega_i$, and $H_i$ is a point stabiliser in $G_i$ for this action. 

At the other extreme, suppose that the components $\mathcal O_i$ are all isomorphic, or equivalently the point stabilisers $\Delta_i$ are conjugate to each other. Then $C$ is the wreath product $C_i\wr {\rm Sym}(I)$ of $C_i$ by ${\rm Sym}(I)$. This is a semidirect product, in which the normal subgroup is the cartesian product of the mutually isomorphic groups $C_i\;(\cong N_{G_i}(H_i)\cong N_{\Gamma}(\Delta_i)/\Delta_i)$ for $i\in I$, and the complement is ${\rm Sym}(I)$, acting on this normal subgroup by permuting the factors $C_i$ via isomorphisms between them.

We can now describe the general form of $C$ by combining these two constructions. We first partition the set of components of $\mathcal O$ into maximal sets $\{\mathcal O_{ij}|i\in I_j\}\;(j\in J)$ of mutually isomorphic objects $\mathcal O_{ij}$, each subset indexed by a set $I_j$. We then define $C_{ij}$ $(\cong N_{G_{ij}}(H_{ij})\cong N_{\Gamma}(\Delta_{ij})/\Delta_{ij}$ with obvious notation) to be the centraliser of $\Gamma$ in ${\rm Sym}(\Omega_{ij})$. Then $C$, and hence also ${\rm Aut}_{\mathfrak C}({\mathcal O})$, is the cartesian product over all $j\in J$ of the wreath products $C_{ij}\wr {\rm Sym}(I_j)$ where $i\in I_j$. This is again a semidirect product, where the normal subgroup is the cartesian product of all the groups $C_{ij}\;(i\in I_j, j\in J)$, and the complement is the cartesian product of the groups ${\rm Sym}(I_j)\;(j\in J)$, each factor ${\rm Sym}(I_j)$ of the latter acting on the normal subgroup by permuting the factors $C_{ij}$ for $i\in I_j$ while fixing all other factors.

This description can be used to determine the cardinality of ${\rm Aut}_{\mathfrak C}({\mathcal O})$. We will restrict attention to categories for which the parent group $\Gamma$ is countable (for instance, if it is finitely generated), since this condition is satisfied by most of the examples studied; the modifications required for an uncountable parent group are straightforward. By  Theorem~\ref{isothm} this implies that ${\rm Aut}_{\mathfrak C}({\mathcal O})$ is also countable for each connected object $\mathcal O$. Recall that a cartesian product of infinitely many non-trivial groups is uncountable, as is the symmetric group on any infinite set. Then we have the following:

\begin{theo}
Let $\mathfrak C$ be a permutational category for which the parent group $\Gamma$ is countable, and let $\mathcal O$ be an object in $\mathfrak C$ with connected components ${\mathcal O}_{ij}$, indexed by sets $I_j\;(j\in J)$ as above. Then
\begin{enumerate}
\item $|{\rm Aut}_{\mathfrak C}({\mathcal O})|>{\aleph_0}$ if and only if either $\mathcal O$ has infinitely many components ${\mathcal O}_{ij}$ with $|{\rm Aut}_{\mathfrak C}({\mathcal O}_{ij})|>1$, or at least one set $I_j$ is infinite;
\item $|{\rm Aut}_{\mathfrak C}({\mathcal O})|=\aleph_0$ if and only if $\mathcal O$ has only finitely many components ${\mathcal O}_{ij}$ with $|{\rm Aut}_{\mathfrak C}({\mathcal O}_{ij})|>1$, each set $I_j$ is finite, and ${\rm Aut}_{\mathfrak C}(\mathcal O_{ij})$ is infinite for some component ${\mathcal O}_{ij}$;
\item $|{\rm Aut}_{\mathfrak C}({\mathcal O})|<\aleph_0$ if and only if $\mathcal O$ has only finitely many components ${\mathcal O}_{ij}$ with $|{\rm Aut}_{\mathfrak C}({\mathcal O}_{ij})|>1$, each set $I_j$ is finite, and ${\rm Aut}_{\mathfrak C}(\mathcal O_{ij})$ is finite for each component ${\mathcal O}_{ij}$.
\end{enumerate}
\end{theo}

\begin{cor}
In Case~(3), when ${\rm Aut}_{\mathfrak C}({\mathcal O})$ is finite, it has order
\[\prod_{j\in J} |{\rm Aut}_{\mathfrak C}(\mathcal O_{ij})|^{|I_j|}|I_j|!\,.\]
\end{cor}

\begin{exam}
 Let $\mathfrak C=\mathfrak M^+$, the category of oriented maps, which has parent group
$\Gamma=\Delta(\infty,2,\infty)=\langle X, Y\mid Y^2=1\rangle$.
For each integer $n\ge 2$ let $\tilde{\mathcal M_n}$ be the minimal regular cover of the map ${\mathcal M}_n\in{\mathfrak M}^+$ in Figure~\ref{MnHn}. This is a regular oriented map with automorphism and monodromy group $S_n$ (in its regular representation) corresponding to the epimorphism $\Gamma\to S_n, X\mapsto (1,2,\ldots, n), Y\mapsto (1,2)$. If we take $\mathcal M$ to be the disjoint union of these maps $\tilde{\mathcal M}_n$, then ${\rm Aut}_{\mathfrak M^+}(\mathcal M)$ is the cartesian product $\prod_{n\ge 2}S_n$. This uncountable group is very rich in subgroups: for instance, every finitely generated residually finite group (such as any finitely generated linear group, by Mal'cev's Theorem~\cite{Mal}) can be embedded in a cartesian product of finite groups of distinct orders, and hence (by Cayley's Theorem) can be embedded in ${\rm Aut}_{\mathfrak M^+}(\mathcal M)$.
\end{exam}


\end{document}